\documentclass[10pt,twoside]{siamart1116}

\usepackage[T1]{fontenc}

\usepackage[english]{babel}
\usepackage{graphicx,epstopdf,epsfig}
\usepackage{amsfonts,epsfig,fancyhdr,graphics, hyperref,amsmath,amssymb}

\newcommand{\Names}{M.I. Bueno, Susana Furtado, Aelita Klausmeier, Joey Veltri}
\newcommand{\Title}{Linear maps preserving the Lorentz spectrum: The $2\times 2$ case}

\renewcommand{\theequation}{\thesection.\arabic{equation}}
\renewtheorem{theorem}{Theorem}[section]

\begin{document}

\bibliographystyle{plain}

%  Leave these commented lines here
%\input{ELAheader-template.tex}
% ELA insert correct page number
\setcounter{page}{1}

\thispagestyle{empty}

%Insert the title of the paper
 \title{\Title}
 
  \author{
M. I. Bueno\thanks{Department of Mathematics,
 University of California Santa Barbara, Santa Barbara, CA 93106, USA
(mbueno@ucsb.edu). The work of the first, third, and fourth authors was partially supported by  the NSF grant DMS-1850663. This publication is also part of the ``Proyecto de I+D+i PID2019-106362GB-I00 financiado
por MCIN/AEI/10.13039/501100011033''.}
% Remember to put \and between any two authors
\and
Susana Furtado\thanks{CEAFEL and Faculdade de Economia da Universidade do Porto, Rua Dr. Roberto Frias, 4200-464 Porto, Portugal,  (sbf@fep.up.pt). The work of the second author was partially
supported  by FCT-Funda\c{c}\~{a}o para a Ci\^{e}ncia e Tecnologia,
under project UIDB/04721/2020.}
\and
Aelita Klausmeier\thanks{ Department of Mathematics, University of Michigan, Ann Arbor, MI 48109, USA (aelita@umich.edu)} 
\and 
Joey Veltri\thanks{ Department of Mathematics, Purdue University, West Lafayette, IN 47906, USA (jveltri@purdue.edu)} }

\markboth{\Names}{\Title}

\maketitle

\begin{abstract}
In this paper a complete  description of the linear maps
$\phi:W_{n}\rightarrow W_{n}$ that preserve the Lorentz spectrum is given when $n=2$
and $W_{n}$ is the space $M_{n}$ of $n\times n$ real matrices or the subspace
$S_{n}$ of $M_{n}$ formed by the symmetric matrices. In both cases, it has been shown
 that $\phi(A)=PAP^{-1}$ for all $A\in W_{2}$, where $P$ is  a
matrix with a certain structure.  It was also shown that such preservers do not change the nature
of the Lorentz eigenvalues (that is, the fact that they are associated with Lorentz eigenvectors
in the interior or on the boundary of the Lorentz cone).  These results extend to $n=2$ those for $n\geq 3$
obtained by  Bueno, Furtado, and Sivakumar (2021). The case $n=2$ has some
specificities, when compared to the case $n\geq3,$ due to the fact that the
Lorentz cone in $\mathbb{R}^{2}$ is polyedral, contrary to what happens when
it is contained in $\mathbb{R}^{n}$ with $n\geq3.$ Thus, the study of the Lorentz spectrum
preservers on $W_n = M_n$ also follows from the known description of the Pareto spectrum
preservers on $M_n$. 
\end{abstract}


\section{Introduction}Given a matrix $A$ in $M_{n},$ the algebra of $n\times n$ matrices with real
entries, and a closed convex cone $K\subseteq\mathbb{R}^{n}$, the
\textit{eigenvalue complementarity problem} consists of finding a scalar
$\lambda\in\mathbb{R}$ and a nonzero vector $x\in\mathbb{R}^{n}$ such that
\[
x\in K,\quad Ax-\lambda x\in K^{\ast},\quad x^{T}(A-\lambda I_{n})x=0,
\]
where
\[
K^{\ast}:=\{y\in\mathbb{R}^{n}:x^{T}y\geq0,\;\forall x\in K\}
\]
denotes the (positive) dual cone of $K$. If $K=\mathbb{R}^{n}$, then the
eigenvalue complementarity problem reduces to the usual eigenvalue problem for
the matrix $A$. 

The eigenvalue complementarity problem originally arose in the
solution of a contact problem in mechanics and has since been used in
other applications in physics, economics, and engineering, including, for example, the stability of
dynamical systems \cite{martins2}.

\indent In this work we consider the complementarity eigenvalue problem associated with the Lorentz cone, defined, for $n \geq 2$, by
\[
\mathcal{K}^{n}:=\{(x,x_{n})\in\mathbb{R}^{n-1}\times\mathbb{R}: ||x||\leq
x_{n}\},
\]
also known as the ice-cream cone. By $||x||$ we denote the $2$-norm of $x.$ If
$n$ is clear from the context, we may simply write $\mathcal{K}$ instead of
$\mathcal{K}^{n}$. The Lorentz cone is widely used in
optimization theory as an instance of a second-order cone, which has special
importance in linear and quadratic programming \cite{zadgold}. 

It is well known that the Lorentz cone is self-dual, that is, $(\mathcal{K}%
^{n})^{\ast}=\mathcal{K}^{n}.$ Therefore, for $A\in M_{n}$, the eigenvalue
complementarity problem relative to $\mathcal{K}^{n}$ consists of finding a
scalar $\lambda\in\mathbb{R}$ and a nonzero vector $x\in\mathbb{R}^{n}$ such
that
\begin{equation}
x\in\mathcal{K}^{n},\quad(A-\lambda I)x\in\mathcal{K}^{n},\quad x^{T}%
(A-\lambda I)x=0, \label{Lorentz1}%
\end{equation}
where, here and throughout, $I$ denotes the identity matrix of the appropriate order. By Corollary
2.1 in \cite{Seeger1}, it is guaranteed that \eqref{Lorentz1} always admits a solution.

If a scalar $\lambda$ and a nonzero vector $x$ satisfy \eqref{Lorentz1}, we call
$\lambda$ a \emph{Lorentz eigenvalue} of $A$ and $x$ an associated
\emph{Lorentz eigenvector} of $A$. We call the set of all Lorentz eigenvalues
of $A$ the \emph{Lorentz spectrum of $A$} and denote it by $\sigma
_{\mathcal{K}}(A)$. For brevity, we write L-eigenvalue, L-eigenvector, and
L-spectrum instead of Lorentz eigenvalue, Lorentz eigenvector, and Lorentz
spectrum, respectively. We classify the L-eigenvalues of a matrix $A\in M_{n}$
by whether they correspond to L-eigenvectors in the interior or on the
boundary of the Lorentz cone. In the first case, we call them \emph{interior}
\emph{L-eigenvalues}, and in the second case, we call them \emph{boundary
L-eigenvalues}. We denote the set of interior L-eigenvalues by $\sigma
_{\mathcal{K}}^{int}(A)$ and the set of boundary L-eigenvalues by
$\sigma_{\mathcal{K}}^{bd}(A)$.

The roots of the characteristic polynomial of a matrix $A\in M_{n}$ will be
called the \emph{standard eigenvalues} of $A,$ to distinguish them from the L-eigenvalues.

In \cite{Bueno2021} the authors focused on the problem of studying the
linear maps $\phi:W_{n}\rightarrow W_{n}$ that preserve the L-spectrum, that
is, such that $\sigma_{\mathcal{K}}(\phi(A))=\sigma_{\mathcal{K}}(A),$ for all
$A\in W_{n}$, where $W_{n}$ is a subspace of $M_{n}$ and $n\geq 3.$ The authors
started by
characterizing such maps $\phi$ for the following subspaces $W_{n}$ of
$M_{n}$: the subspace of diagonal matrices; the subspace of block-diagonal
matrices $\widetilde{A}\oplus\lbrack a]$, where $\widetilde{A}\in M_{n-1}$ is
symmetric; and the subspace of block-diagonal matrices $\widetilde{A}%
\oplus\lbrack a]$, where $\widetilde{A}\in M_{n-1}$ is a generic matrix. In
each of these cases, it was shown that the maps should be what were called
\emph{standard maps}, that is, maps of the form $\phi(A)=PAQ$ for all $A\in
W_{n}$ or $\phi(A)=PA^{T}Q$ for all $A\in W_{n},$ for some matrices $P,Q\in
M_{n}$.  However, when $W_{n}$ is either $M_{n}$ or the subspace $S_{n}$ of
symmetric matrices in $M_{n},$ just
%In addition, when $W_{n}$ is either $M_{n}$ or the subspace $S_{n}$
%of symmetric matrices in $M_{n},$ 
the standard linear maps $\phi
:W_{n}\rightarrow W_{n}$  that preserve the
L-spectrum were described, and it was conjectured that linear maps that are not
standard do not preserve the L-spectrum. (See also the recent paper
\cite{Seeger3} in which the linear preservers $\phi:M_{n}\rightarrow M_{n}$
are investigated.)

The goal of this paper is to consider the case $n=2.$ The main differentiating feature
between the cases $n\geq 3$ and $n=2$ is that the Lorentz cone in $\mathbb{R}^{2}$
is \emph{polyhedral}, i.e., it can be expressed as the intersection of a
finite number of half-spaces. This implies that the L-spectrum of a matrix in
$M_{2}$ is always finite, contrary to what happens for matrices of order
$n\geq 3$, which can have infinite L-spectrum.

To our knowledge, the only
polyhedral cone whose spectral linear preservers on $M_n$ have been studied in depth in
the literature is the Pareto cone \cite{zadsha}. It can be easily verified that, for $n=2,$ the Pareto cone
is a clockwise rotation of the Lorentz cone by an  angle of $\frac{\pi}{4}.$
So, a description of the L-spectrum preservers on $M_{2}$ follows from the one
of the Pareto spectrum preservers on $M_{n}$ given in \cite{zadsha}, taking
$n=2,$ and reciprocally. 

In this paper, we give an independent proof of the characterization
of the linear maps $\phi:W_{2}\rightarrow W_{2}$ that preserve the L-spectrum
when $W_{2}=M_{2}$ and, in addition, consider the new case $W_{2}=S_{2},$ the
subspace of $M_{2}$ of symmetric matrices$.$ Also, for both cases of $W_{2}$
we show that the nature of the L-eigenvalues (being associated with
L-eigenvectors in the interior or on the boundary of the Lorentz cone) is not
changed by the L-spectrum preservers. To prove our results, we introduce
techniques that explore the knowledge of the Lorentz spectrum of matrices in
$M_{2}$ and hope that the ideas behind our proofs can be extended to complete
the study of the L-spectrum preservers studied in \cite{Bueno2021} for
matrices in $M_{n}$, with $n\geq 3$, in which case no connection exists with the
Pareto cone. It follows from our characterization that such preservers on
$M_{2}$ are standard and that, in the case $W_{2}=M_{2}$, their form is less
restrictive than the one for $n\geq3$. (See Theorem \ref{rr} where the result
for $n\geq3$ is recalled.)

We next give the main results of this paper. Recall that $M_2$ denotes the space of $2\times 2$ real matrices and $S_2$ denotes the subspace of $M_2$ of symmetric matrices. 

\begin{theorem}
\label{tmain1} Let  $\phi:W_{2}\rightarrow W_{2}$ be a linear map, with
$W_{2}\in\{M_{2},S_{2}\}$. Then, $\phi$ preserves the L-spectrum if and
only if $\phi(A)=PAP^{-1}$ for all $A\in W_{2}$, or $\phi(A)=QAQ^{-1}$
for all $A\in W_{2}$, where
\begin{equation}
P=\left[
\begin{array}
[c]{cc}%
\alpha & \beta\\
\beta & \alpha
\end{array}
\right] \quad \text{ and } \quad Q=\left[
\begin{array}
[c]{cc}%
-\alpha & -\beta\\
\beta & \alpha
\end{array}
\right]  ,\label{PQ}%
\end{equation}
for some $\alpha,\beta\in\mathbb{R}$ with $\alpha^{2}-\beta^{2}=1$, and $\beta=0$ if $W_{2}=S_{2}.$
\end{theorem}

\begin{corollary}
\label{cornature}Let $\phi:W_{2}\rightarrow W_{2}$ be a linear map. If $\phi$
preserves the L-spectrum, then, for all $A\in W_{2}$,
\[
\sigma_{\mathcal{K}}^{int}(A)=\sigma_{\mathcal{K}}^{int}(\phi(A))\quad
\text{and}\quad\sigma_{\mathcal{K}}^{bd}(A)=\sigma_{\mathcal{K}}^{bd}%
(\phi(A)).
\]

\end{corollary}

The paper is organized as follows. In Section \ref{background} we introduce
some known results in the literature  regarding the
L-spectrum of a matrix $A\in M_{n}$ and its linear preservers. In Section
\ref{spectrum} we obtain a description of the L-eigenvalues of a generic
matrix in $M_{2}$ and give some related results that will be helpful in the
proof of Theorem \ref{tmain1}. In Section \ref{basis} we deduce some
conditions that should be satisfied by the images of matrices in certain bases
for $S_{2}$ and $M_{2}$, respectively, under an L-spectrum linear preserver.
Finally, in Section \ref{proofmain}, we prove Theorem \ref{tmain1} and
Corollary \ref{cornature}. We
conclude the paper with some final remarks in Section \ref{conclusions}.

\section{Background}

\label{background}

In this section we present some results known in the literature concerning the
characterization of the L-spectrum of a matrix in $M_n$, and properties of linear
preservers of the L-spectrum. We also introduce some related useful concepts
and notation.

\subsection{L-spectrum of a matrix}

We first observe that
\[
\sigma_{\mathcal{K}}(A)=\sigma_{\mathcal{K}}^{int}(A)\cup\sigma_{\mathcal{K}%
}^{bd}(A),
\]
where this union is not necessarily disjoint. (Recall the definitions of interior and boundary eigenvalues in the introduction.)

We also note that any L-eigenvector $[x$ $x_{n}]^{T}$ of $A\in M_{n}$, with $x_n\in \mathbb{R}$,  can be
normalized to have $x_{n}=1$ while remaining in the Lorentz cone. Such a
normalized L-eigenvector corresponds to an interior L-eigenvalue if $||x||<1$
and to a boundary L-eigenvalue if $||x||=1$.

The next characterization of interior and boundary L-eigenvalues of a matrix
$A\in M_{n}$ is known \cite{Seeger2}. 

\begin{proposition}
\label{intchar} Let $A\in M_{n}.$ Then,

\begin{enumerate}
\item $\lambda$ is an interior L-eigenvalue of $A$ if and only if $\lambda$ is
a standard eigenvalue of $A$ associated with an eigenvector in the interior of
$\mathcal{K}^{n}$.

\item $\lambda$ is a boundary L-eigenvalue of $A$ if and only if there is some
$s\geq0$ and a vector $x\in \mathbb{R}^{n-1}$, with  $||x||=1$, such that
\[
(A-\lambda I)\left[
\begin{array}
[c]{c}%
x\\
1
\end{array}
\right]  =s\left[
\begin{array}
[c]{c}%
-x\\
1
\end{array}
\right]  .
\]

\end{enumerate}
\end{proposition}

From Proposition \ref{intchar}, we have the following useful observation.

\begin{corollary}
\label{intA-A} \label{l1}Let $A\in M_{n}.$ Then, $\lambda\in\sigma
_{\mathcal{K}}^{int}(A)$ if and only if $-\lambda\in\sigma_{\mathcal{K}}%
^{int}(-A)$.
\end{corollary}

In contrast with interior L-eigenvalues, a boundary L-eigenvalue may or may
not  be a standard eigenvalue. A surprising fact, compared with the
classical eigenvalue problem, is that a matrix may have infinitely many
boundary L-eigenvalues, though this does not occur in the $2\times2$ case
since the Lorentz cone for $n=2$ is  a polyhedral cone. (See
\cite{Seeger2} for a proof that there are only finitely many complementarity
eigenvalues relative to a polyhedral cone.)

\subsection{Linear preservers of the L-spectrum}

In \cite{Bueno2021} the following important result was shown for matrices of
size $n\geq3$, although the presented proof is also valid for $2\times2$
matrices. By $W_n$ we denote any of the spaces $M_n$ or  $S_n$, the subspace of symmetric matrices.

\begin{proposition}
\cite{Bueno2021} \label{bijunit} Let $n\geq2$. If  $\phi:W_{n}\rightarrow W_{n}$ is a linear map preserving the
L-spectrum, then $\phi$ is bijective and $\phi(I)=I$.
\end{proposition}

An immediate consequence of Proposition \ref{bijunit} is that if $\phi
:W_{n}\rightarrow W_{n}$ is a linear map preserving the L-spectrum, then
$\phi^{-1}$ also preserves the L-spectrum.

\vspace{0.3cm}
For completeness and for  purpose of comparison with our main result, Theorem \ref{tmain1}, we next state the characterization obtained in \cite{Bueno2021} of the standard linear maps $\phi:W_{n}\rightarrow W_{n}$ that
preserve the L-spectum, when $n\geq3$.

\begin{theorem}
\cite{Bueno2021} \label{rr} Let $n\geq3$ and let $\phi:W_{n}\rightarrow W_{n}$ be a standard
map. Then, $\phi$ preserves the L-spectrum if and only if there exists an
orthogonal matrix $Q\in M_{n-1}$ such that
\[
\phi(A)=(Q\oplus\lbrack1])A(Q^{T}\oplus\lbrack1]),
\]
for all $A\in W_{n}$.

\end{theorem}

\section[L-spectrum of 2x2 matrices]{L-spectrum of $2\times2$ matrices}

\label{spectrum}

In the next theorem we present a characterization of the L-eigenvalues of
$2\times2$ matrices and then we give some related properties.

\begin{theorem}
\label{lcs2x2}Let
\begin{equation}
A=%
\begin{bmatrix}
a & b\\
c & d
\end{bmatrix}
\in M_{2}. \label{A}%
\end{equation} Then,

\begin{enumerate}
\item $a$ is an interior L-eigenvalue of $A$ if and only if $b=0$ and either
$a=d$ or $|a-d|<|c|$;

\item $\lambda\in\mathbb{R}\setminus\{a\}$ is an interior L-eigenvalue of $A$
if and only if \\$\lambda\in\left\{  \frac{a+d\pm\sqrt{(a-d)^{2}+4bc}}%
{2}\right\}\subseteq \mathbb{R}  $ and $|b|<|a-\lambda|$;

\item $\lambda$ is a boundary L-eigenvalue of $A$ if and only if one of the
following holds:

\begin{enumerate}
\item $\lambda=\frac{(a+d)+(b+c)}{2}$ and $a-d \leq c -b$,

\item $\lambda=\frac{(a+d)-(b+c)}{2}$ and $a-d\leq b-c$.
\end{enumerate}
\end{enumerate}
\end{theorem}

\begin{proof}
 Conditions 1 and 2 follow immediately from the fact that, by Proposition \ref{intchar}, $\lambda$ is an
interior L-eigenvalue of $A$ if and only if there is some $x\in\mathbb{R},$
with $|x|<1,$ such that
\begin{equation}
0=(A-\lambda I)\left[
\begin{array}
[c]{c}%
x\\
1
\end{array}
\right]  =\left[
\begin{array}
[c]{c}%
(a-\lambda)x+b\\
cx+(d-\lambda)
\end{array}
\right]  . \label{k}%
\end{equation}
Now we show Condition 3. By Proposition \ref{intchar}, we have that $\lambda$
is a boundary L-eigenvalue of $A$ if and only if there is some $s\geq0$ and
$x\in\{-1,1\}$ such that
\[
\left[
\begin{array}
[c]{cc}%
a-\lambda & b\\
c & d-\lambda
\end{array}
\right]  \left[
\begin{array}
[c]{c}%
x\\
1
\end{array}
\right]  =s\left[
\begin{array}
[c]{c}%
-x\\
1
\end{array}
\right]  \ \Leftrightarrow\ \left[
\begin{array}
[c]{c}%
(a-\lambda+s)x+b\\
cx+(d-\lambda-s)
\end{array}
\right]  =0.
\]
When $x=1$, this is equivalent to
\[
\
\begin{array}
[c]{c}%
\ \left\{
\begin{array}
[c]{c}%
\lambda=a+b+s\\
\lambda=c+d-s
\end{array}
\right.  \text{ for some }s\geq0,\\
\end{array}
\]
that is, 
\[
\lambda=\frac{a+b+c+d}{2}\quad\text{ and }\quad a-d\leq c-b.
\]
When $x=-1$, we get
\[
\
\begin{array}
[c]{c}%
\left\{
\begin{array}
[c]{c}%
\lambda=a-b+s\\
\lambda=d-c-s
\end{array}
\right.  \text{ for some }s\geq0,
\end{array}
\]
that is,
\[
\lambda=\frac{a+d-b-c}{2}\quad\text{ and }\quad a-d\leq b-c.
\]

\end{proof}

Based on the characterization of the boundary L-eigenvalues of a matrix in
$M_{2}$ given in Theorem \ref{lcs2x2}, we introduce the following definitions.

\begin{definition}
Let $A\in M_{2}$. We say that $\lambda$ is a \emph{type }$+$\emph{ boundary L-eigenvalue} of $A$
(resp. a \emph{type }$-$\emph{ boundary L-eigenvalue} of $A$) if Condition 3a
(resp. Condition 3b) in Theorem \ref{lcs2x2} holds.

Moreover, we say that a boundary L-eigenvalue $\lambda$ of $A$ is
\emph{strict} if $\lambda$ is of type $+$ and $a-d<c-b$, or if $\lambda$ is of
type $-$ and $a-d<b-c$. If $\lambda$ is a boundary L-eigenvalue of both type
$+$ and type $-,$ then $\lambda$ is strict if at least one of the previous
strict inequalities holds.
\end{definition}

We next present some immediate consequences of Theorem \ref{lcs2x2}. We first
introduce two useful concepts.

\begin{definition}
Let $A\in M_{2}$ be as in (\ref{A}). The \emph{trace} of $A,$ denoted by
$\operatorname*{tr}(A),$ is the sum of the diagonal entries of $A,$ that is,
$\operatorname*{tr}(A)=a+d.$ The \emph{anti-trace} of $A,$ denoted by
$\operatorname*{antitr}(A),$ is the sum of the anti-diagonal entries of $A,$
that is, $\operatorname*{antitr}(A)=b+c.$
\end{definition}

\begin{corollary}
\label{2bd} Let $A\in M_{2}$. If $A$ has a type $+$ boundary L-eigenvalue
$\lambda_{1}$ and a type $-$ boundary L-eigenvalue $\lambda_{2}$, then

\begin{enumerate}
\item $\lambda_{1}+\lambda_{2}=\operatorname*{tr}(A).$

\item $|\lambda_{1}-\lambda_{2}|=|\operatorname*{antitr}(A)|$.
\end{enumerate}
\end{corollary}

\begin{corollary}
\label{non-strict-classic} Let $A\in M_{2}$ be as in (\ref{A}) and let $\lambda$ be a boundary L-eigenvalue of $A$. Then, 
$\lambda$ is a standard eigenvalue of $A$ if and only if $A$ has a non-strict boundary L-eigenvalue. 
\end{corollary}

\begin{proof}
By Theorem \ref{lcs2x2}, if $\lambda$ is a
type $+$ boundary L-eigenvalue of $A$, then
\[
\lambda=\frac{a+d+b+c}{2}\quad\text{ and }\quad a-d\leq c-b,
\]
and if $\lambda$ is a type $-$ boundary L-eigenvalue of $A$, then
\[
\lambda=\frac{a+d-b-c}{2}\quad\text{ and }\quad a-d\leq b-c.
\]
An elementary calculation shows that, in any case,
\[
\det(A-\lambda I)=\frac{1}{4}\left(  (b-c)^{2}%
-(a-d)^{2}\right),
\]
which is zero if and only if $|a-d|=|b-c|.$ Thus, the  claim follows.
\end{proof}

The next result says that if we change the signs of both $b$ and $c$ in a matrix $A$ as in \eqref{A}, then the interior and the boundary L-eigenvalues of $A$ get preserved. 
\begin{corollary}
\label{rr1} Let $A\in M_{2}$ and $B=TAT$, where
\begin{equation}
T=\left[  -1\right]  \oplus\left[  1\right]  . \label{P}%
\end{equation}
Then $A$ and $B$ have the same L-spectrum. Moreover, we have $\sigma
_{\mathcal{K}}^{int}(A)=\sigma_{\mathcal{K}}^{int}(B)$ and
$\sigma_{\mathcal{K}}^{bd}(A)=\sigma_{\mathcal{K}}^{bd}(B).$
Additionally, $\lambda$ is a type $+$ boundary L-eigenvalue of $A$ if and only
if $\lambda$ is a type $-$ boundary L-eigenvalue of $B.$
\end{corollary}

By using Theorem \ref{lcs2x2}, we next give the explicit L-spectrum of
the  matrices in a basis of  $M_2$ and $S_2$, which  will be  used in the characterization of the
linear maps preserving the L-spectrum. In each case, the L-spectrum is
presented as the union of two sets, namely, $\sigma_{\mathcal{K}}^{int}%
(A)\cup\sigma_{\mathcal{K}}^{bd}(A)$. Here and throughout, for $i,j\in
\{1,2\}$, $E_{ij}$ denotes the $2\times2$ matrix with all entries $0$ except
the one in position $(i,j)$ which is $1.$

\begin{corollary}
\label{spect-cases} We have
\begin{itemize}
\item $\sigma_{\mathcal{K}}(E_{11})=\{0\}\cup\emptyset$

\item $\sigma_{\mathcal{K}}(E_{21})=\{0\}\cup\{1/2\}$

\item $\sigma_{\mathcal{K}}(E_{22})=\{1\}\cup\{1/2\}$

\item $\sigma_{\mathcal{K}}(E_{12}+E_{21})=\emptyset\cup\{-1,1\}$
\end{itemize}
\end{corollary}

\section{Images of matrices in a basis of $W_{2}$ under an L-spectrum
preserver}

\label{basis}

Let us consider a linear map $\phi:W_{2}\rightarrow W_{2}$ preserving the
L-spectrum, with $W_2 \in \{M_2, S_2\}$. In this section we obtain a generic form that $\phi(A)$ should
have when $A$ is a matrix in a specific basis of $W_{2},$ namely, the
basis $\{E_{11},E_{22},E_{12}+E_{21}\}$ if $W_{2}=S_{2}$, and the basis
$\{E_{11},E_{22},E_{21},E_{12}+E_{21}\}$ if $W_{2}=M_{2}.$ For $E_{12}%
+E_{21},$ the possible images under $\phi$ are exactly determined.

We begin with a result which shows that under certain conditions, a linear
preserver of the L-spectrum preserves the interior and boundary L-eigenvalues.
This will be key in proving the remaining results.

\begin{lemma}
\label{key-result} Let $\phi:W_{2}\rightarrow W_{2}$ be a linear map that
preserves the L-spectrum. If $A\in W_{2}$ has two distinct strict
boundary L-eigenvalues,
then
\begin{equation}
\sigma_{\mathcal{K}}^{int}(A)=\sigma_{\mathcal{K}}^{int}(\phi(A))\neq
\emptyset\quad\text{ and }\quad\sigma_{\mathcal{K}}^{bd}(A)=\sigma
_{\mathcal{K}}^{bd}(\phi(A)). \label{ds}%
\end{equation}

\end{lemma}

\begin{proof}
Let $A$ be as in (\ref{A}). Since $A$ has two distinct strict boundary
L-eigenvalues, say $\lambda_1$ and $\lambda_2$, by Theorem \ref{lcs2x2} we have $a-d<c-b$ and $a-d<b-c$. This
implies that $-A$ does not have any boundary L-eigenvalues and, consequently,
has at least one interior L-eigenvalue since every matrix has a nonempty
L-spectrum. Hence, we have
\[
\sigma_{\mathcal{K}}^{bd}(A)=\{\lambda_{1},\lambda_{2}\},\quad\sigma
_{\mathcal{K}}^{int}(-A)\neq\emptyset,\quad\text{and}\quad\sigma_{\mathcal{K}%
}^{bd}(-A)=\emptyset.
\]
Taking into account Corollary \ref{intA-A} and the fact that, by Corollary
\ref{non-strict-classic}, $\lambda_{1}$ and $\lambda_{2}$ are not standard
eigenvalues of $A$, we have
\[
\sigma_{\mathcal{K}}^{int}(A)=-\sigma_{\mathcal{K}}^{int}(-A),\quad
\sigma_{\mathcal{K}}^{int}(A)\neq\emptyset,\quad\text{and}\quad
\sigma_{\mathcal{K}}^{int}(A)\cap\{\lambda_{1},\lambda_{2}\}=\emptyset.
\]
Since $\phi$ preserves the L-spectrum, for $i\in \{1,2\}$ we should have $\lambda
_{i}\in\sigma_{\mathcal{K}}^{bd}(\phi(A)),$ as otherwise $\lambda_{i}\in
\sigma_{\mathcal{K}}^{int}(\phi(A))$, which implies, by Corollary \ref{intA-A}, that  $-\lambda_{i}%
\in\sigma_{\mathcal{K}}^{int}(\phi(-A)),$ a contradiction since $-\lambda_{i}$
is not an L-eigenvalue of $-A.$ Then, since $\phi(A)$ has two boundary
L-eigenvalues, which are the boundary L-eigenvalues of $A,$ it follows that
the interior L-eigenvalues of $A$ are also interior L-eigenvalues of
$\phi(A)$.
\end{proof}

Before we fulfill the main purpose of this section, we state a simple
consequence of Lemma \ref{key-result} that will be used in the proof of
Theorem \ref{tmain1} in the next section.

\begin{lemma}
\label{singular} Let $\phi:W_{2}\rightarrow W_{2}$ be a linear map that
preserves the L-spectrum. Then, $\phi(E_{11}+E_{21})$ is singular.
\end{lemma}

\begin{proof}
Let $\varepsilon>0\ $and $A_{\varepsilon}:=(-1-\varepsilon)E_{11}-E_{21}.$ The
matrix $A_{\varepsilon}$ has two distinct strict boundary L-eigenvalues,
implying, by Lemma \ref{key-result}, that $\phi(A_{\varepsilon})$ has the same
interior L-eigenvalues as $A_{\varepsilon}.$ Since $0$ is an interior L-eigenvalue of
$A_{\varepsilon}$, $\phi(A_{\varepsilon})$ is singular. By continuity,
$\phi(-E_{11}-E_{21})$ is singular, and hence, so is $\phi(E_{11}+E_{21})$.
\end{proof}

\subsection{Necessary forms for the images of a basis}

\begin{lemma}
\label{E22-E11} Let $\phi:W_{2}\rightarrow W_{2}$ be a linear map that
preserves the L-spectrum. Then,
\[
\phi(E_{11})=%
\begin{bmatrix}
1-a & \mp\sqrt{a^{2}-a}\\
\pm\sqrt{a^{2}-a} & a
\end{bmatrix}
,\quad\phi(E_{22})=%
\begin{bmatrix}
a & \pm\sqrt{a^{2}-a}\\
\mp\sqrt{a^{2}-a} & 1-a
\end{bmatrix}
\]
for some $a\leq0,$ and
\[
\phi(E_{12}+E_{21})=%
\begin{bmatrix}
m & r\\
-r \pm2 & -m
\end{bmatrix},
\]
for some $m,r\in\mathbb{R}.$ In particular, if $W_{2}=S_{2},$ then%
\[
\phi(E_{11})=E_{11},\quad\phi(E_{22})=E_{22},
\]
and
\[
\phi(E_{12}+E_{21})=%
\begin{bmatrix}
m & r\\
r & -m
\end{bmatrix},
\]
for some $m\in\mathbb{R}$ and $r \in\{-1,1\}.$
\end{lemma}

\begin{proof}
For $\varepsilon\in\mathbb{R}\setminus\{0\},$ let $G_{\varepsilon}%
:=E_{22}+\varepsilon(E_{12}+E_{21})$, whose standard eigenvalues are
$(1\pm\sqrt{1+4\varepsilon^{2}})/2$. By Theorem \ref{lcs2x2},
\[
\sigma_{\mathcal{K}}^{int}(G_{\varepsilon})=\left\{  \frac{1+\sqrt
{1+4\varepsilon^{2}}}{2}\right\}  \text{ and }\sigma_{\mathcal{K}}%
^{bd}(G_{\varepsilon})=\left\{  \frac{1}{2}\pm\varepsilon\right\},
\]
and both boundary L-eigenvalues are strict. Thus, by Lemma \ref{key-result},
(\ref{ds}) holds with $A$ replaced by $G_{\varepsilon}$. Let
\[
\phi(E_{22}):=%
\begin{bmatrix}
a & b\\
c & d
\end{bmatrix}
\text{ and }\phi(E_{12}+E_{21}):=%
\begin{bmatrix}
m & r\\
p & q
\end{bmatrix}.
\]
Then, by Corollary \ref{2bd} applied to  $\phi(G_{\varepsilon})$,
\[
a+d+\varepsilon(m+q)=1,\quad b+c+\varepsilon(r+p)=\pm2\varepsilon.
\]
Since $\varepsilon\neq 0$ is arbitrary, we have
\[
a+d=1,\quad m+q=0,\quad b+c=0,\quad r+p=\pm2.
\]
Hence,
\[
\phi(E_{22})=%
\begin{bmatrix}
a & b\\
-b & 1-a
\end{bmatrix}
\text{ and}\quad\phi(E_{12}+E_{21})=%
\begin{bmatrix}
m & r\\
- r\pm2 & -m
\end{bmatrix}
.
\]
From the obtained form of $\phi(E_{22}),$ we conclude, by Theorem
\ref{lcs2x2}, that $1$ is not a boundary L-eigenvalue of $\phi(E_{22}).$ Since
$\sigma_{\mathcal{K}}(\phi(E_{22}))=\sigma_{\mathcal{K}}(E_{22})=\{1,1/2\},$
it follows that $1$ is an interior L-eigenvalue of $\phi(E_{22})$. This
implies that
\[
\det(\phi(E_{22})-I)=b^{2}-a^{2}+a=0.
\]
By Theorem \ref{lcs2x2}, $b\neq 0$. Moreover, $|b|<|a-1|$, i.e.,   $b^{2}<(a-1)^{2}$. Since
$b^{2}=a(a-1)\geq0$, we get $a\leq0$,
\[
\phi(E_{22})=%
\begin{bmatrix}
a & \pm\sqrt{a^{2}-a}\\
\mp\sqrt{a^{2}-a} & 1-a
\end{bmatrix},
\text{ and}%
\]%
\begin{equation}\label{eqqq}
\phi(E_{11})= \phi(I-E_{22})=I-\phi(E_{22})=%
\begin{bmatrix}
1-a & \mp\sqrt{a^{2}-a}\\
\pm\sqrt{a^{2}-a} & a
\end{bmatrix},
\end{equation}
where the second equality in \eqref{eqqq}  follows from Proposition \ref{bijunit}.

The particular claim in the statement for  $W_{2}=S_{2}$ follows since $\phi(E_{11})$ and $\phi(E_{12}+E_{21})$ are symmetric and $a \leq 0$.
\end{proof}

Notice that, if $\phi: W_2\rightarrow W_2$ is a linear map preserving the $L$-spectrum,  by  Lemma \ref{E22-E11}, $\phi$  preserves the trace of $E_{11}$, $E_{22}$ and
$E_{12}+E_{21},$ and therefore it preserves the trace of all matrices in $S_{2}$. Also,
observe that $\phi$ preserves the modulus of the anti-trace of $E_{11}$,
$E_{22}$, and $E_{12}+E_{21}.$ Moreover, if $\phi$ preserves the anti-trace of
$E_{12}+E_{21}$, then $\phi$ preserves the anti-trace of all matrices in
$S_{2}$; otherwise, the anti-traces of $A$ and $\phi(A)$ have opposite signs
for all $A\in S_{2}.$ %Thus, $\phi:S_{2}\rightarrow S_{2}$ preserves the trace
%and the modulus of the anti-trace. 
These results are contained in the following
corollary and  extended to the case $\phi:M_{2}\rightarrow M_{2}$.

\begin{corollary}
\label{trace} Let $\phi:W_{2}\rightarrow W_{2}$ be a linear map that preserves
the L-spectrum. Then,
\[
\operatorname*{tr}(A)=\operatorname*{tr}(\phi(A))\quad\text{for all }A\in
W_{2},
\]
and either
\[
\operatorname*{antitr}(A)=\operatorname*{antitr}(\phi(A))\quad\text{for all
}A\in W_{2}%
\]
or
\[
\operatorname*{antitr}(A)=-\operatorname*{antitr}(\phi(A))\quad\text{for all
$A\in W_{2}.$}%
\]

\end{corollary}

\begin{proof}
Let $A$ be as in (\ref{A}) and let
\[
\phi(A):=\left[
\begin{array}
[c]{cc}%
r & s\\
p & q
\end{array}
\right].
\]
Let $\delta$ be an arbitrary real number such that
\[
a-d<\delta+c-b,\quad a-d<\delta+b-c,\quad\text{and} \quad b+c\neq2\delta.
\]
Let $A_{\delta}=A+\delta E_{22}-\delta(E_{12}+E_{21})$. Notice that
$A_{\delta}$ has two strict boundary L-eigenvalues, namely
\begin{equation}
\lambda_{1}=\frac{a+d+b+c-\delta}{2}\quad \text{and} \quad \lambda_{2}=\frac{a+d-b-c+3\delta
}{2},\label{cond}%
\end{equation}
which are distinct since $b+c\neq2\delta.$ Thus, by Lemma \ref{key-result},
$\lambda_{1}$ and $\lambda_{2}$ are also boundary L-eigenvalues of
$\phi(A_{\delta})$. Taking into account the form of $\phi(\delta E_{22}%
-\delta(E_{12}+E_{21}))$ that follows from Lemma \ref{E22-E11}, the boundary
L-eigenvalues of $\phi(A_{\delta})$ are
\begin{equation}
\beta_{1}=\frac{r+q+s+p-\delta}{2},\quad\beta_{2}=\frac{r+q-s-p+3\delta}%
{2}\label{d1}%
\end{equation}
if $\text{antitr}(\phi(E_{12}+E_{21}))=2,$ and
\begin{equation}
\beta_{1}=\frac{r+q+s+p+3\delta}{2},\quad\beta_{2}=\frac{r+q-s-p-\delta}%
{2}\label{d2}%
\end{equation}
if $\text{antitr}(\phi(E_{12}+E_{21}))=-2.$ As $\{\lambda_{1}%
,\lambda_{2}\}=\{\beta_{1},\beta_{2}\}$, we have
\[
\lambda_{1}+\lambda_{2}=\beta_{1}+\beta_{2}%
\]
and
\[
\text{ }\lambda_{1}-\lambda_{2}=\beta_{1}-\beta_{2}\quad \text{ or }\quad \lambda
_{1}-\lambda_{2}=-(\beta_{1}-\beta_{2}).
\]
Since $\lambda_{1}+\lambda_{2}=a+d+\delta$ and $\beta_{1}+\beta_{2}%
=r+q+\delta,$ we get $a+d=r+q.$ We also have $\lambda_{1}-\lambda
_{2}=b+c-2\delta$. Moreover, $\beta_{1}-\beta_{2}=s+p-2\delta$ if (\ref{d1}) holds,
and $\beta_{1}-\beta_{2}=s+p+2\delta$ if (\ref{d2}) holds. In the first case,
$\lambda_{1}-\lambda_{2}=-(\beta_{1}-\beta_{2})$ only for $\delta
=\frac{b+c+s+p}{4}.$ Thus, for $\delta\neq\frac{b+c+s+p}{4}$, we have
$\lambda_{1}-\lambda_{2}=\beta_{1}-\beta_{2},$ implying $b+c=s+p.$ In the
second case, $\lambda_{1}-\lambda_{2}=\beta_{1}-\beta_{2}$ only for
$\delta=\frac{b+c-s-p}{4}.$ Thus, for $\delta\neq\frac{b+c-s-p}{4}$, we have
$\lambda_{1}-\lambda_{2}=-(\beta_{1}-\beta_{2}),$ implying $b+c=-(s+p).$ Since
$\delta$ is an arbitrary number satisfying (\ref{cond}), it ranges over an
infinite set, and hence the claim follows. %thus it follows that $b+c=\pm(s+p)$.
\end{proof}

We next describe the generic structure of the image of $E_{21}$ under a linear
map preserving the L-spectrum.

\begin{lemma}
\label{lE21}Let $\phi:M_{2}\rightarrow M_{2}$ be a linear map that preserves
the L-spectrum. Then,
\[
\phi(E_{21})=\left[
\begin{array}
[c]{cc}%
\pm\sqrt{b^{2}+b} & \mp b\\
\pm(b+1) & \mp\sqrt{b^{2}+b}%
\end{array}
\right]  ,\quad b\geq0.
\]

\end{lemma}

\begin{proof}
By Corollary \ref{trace},
\[
\phi(E_{21})=\left[
\begin{array}
[c]{cc}%
a & b\\
-b\pm1 & -a
\end{array}
\right]
\]
for some $a,b\in\mathbb{R}$. By Theorem \ref{lcs2x2}, this implies 
$\sigma_{\mathcal{K}}^{bd}(\phi(E_{21}))\subseteq\{-1/2,1/2\}$. On the other
hand, by Corollary \ref{spect-cases}, $\sigma_{\mathcal{K}}(E_{21})=\{0,1/2\}$ . Thus, since $\phi$ preserves the L-spectrum, $0$ is an interior L-eigenvalue of $\phi(E_{21})$. Hence, by
Theorem \ref{lcs2x2}, either $a=b=0$, or $|b|<|a|$ (i.e., $b^2 < a^2$). Since $\phi(E_{21})$ is
singular, we also have $a^{2}=b^{2}\mp b.$ Thus, $a^{2} = b^{2} + b$ if $b>0$
and $a^{2}= b^{2}-b$ if $b<0$, implying the claim.
\end{proof}

\subsection{Explicit image of $E_{12}+E_{21}$}

The following two lemmas will be used in determining $\phi(E_{12}+E_{21})$ under a linear L-spectrum preserver $\phi$.
By $|| \cdot ||_F$ we denote the Frobenius norm of a matrix. 
\begin{lemma}
\label{lintlim}Let $A\in M_{2}$ be as in (\ref{A}). Suppose $A$ has two
distinct standard real eigenvalues and at least one of them, say $\lambda
_{A},$ is an interior L-eigenvalue$.$ Moreover, suppose that $\lambda_{A}\neq
a$. Then, for any $\varepsilon>0,$ there is some $\delta>0$ such that any
$B\in M_{2}$ with $||B-A||_{F}<\delta$ has an interior L-eigenvalue
$\lambda_{B}$ satisfying $|\lambda_{A}-\lambda_{B}|<\varepsilon.$ That is,
sufficiently small perturbations of $A$ have an interior L-eigenvalue
arbitrarily close to $\lambda_{A}.$
\end{lemma}

\begin{proof}
Suppose that $\lambda_{A}$ is an interior L-eigenvalue of $A.$ By Theorem
\ref{lcs2x2}, since $\lambda_{A}\neq a$, we have $|b|<|a-\lambda_{A}|,$
that is, $b^{2}-(a-\lambda_{A})^{2}<0.$ Since $\lambda_{A}$ depends
continuously on the entries of $A$, any sufficiently small perturbation of $A,$
say%
\[
A_{\varepsilon}:=\left[
\begin{array}
[c]{cc}%
a_{\varepsilon} & b_{\varepsilon}\\
c_{\varepsilon} & d_{\varepsilon}%
\end{array}
\right]  ,
\]
has a real eigenvalue $\lambda_{A}^{\varepsilon}$ arbitrarily close to
$\lambda_{A}$ and such that $\lambda_{A}^{\varepsilon}\neq a_{\varepsilon}$
and $|b_{\varepsilon}|<|a_{\varepsilon}- \lambda_{A}^{\varepsilon}|$. Note that, since $A$ has distinct real eigenvalues, for $\varepsilon$ sufficiently small,  both eigenvalues  of $A_{\varepsilon}$ are also distinct and real.  By Theorem
\ref{lcs2x2}, $\lambda_{A}^{\varepsilon}$ is an interior L-eigenvalue of
$A_{\varepsilon}.$ 
\end{proof}

\begin{lemma}
\label{E1221pert}Let $\lambda\in\{-1,1\}.$ Then, there is some $\varepsilon>0$ such
that, in any neighborhood of $E_{12}+E_{21}$, there is a matrix with no
L-eigenvalue at distance from $\lambda$ smaller than $\varepsilon$.
\end{lemma}

\begin{proof}
Let $H:=E_{12}+E_{21}.$ For any $\delta \in\mathbb{R}$, the matrices
\begin{equation}
H_{\delta}:=H+\delta \left[
\begin{array}
[c]{cc}%
1 & 0\\
0 & -1
\end{array}
\right]
\end{equation}
and  $-H_{\delta}$ have standard eigenvalues $\beta_{1}=-\sqrt
{\delta^{2}+1}$ and $\beta_{2}=\sqrt{\delta^{2}+1}.$ Notice that,
for $i\in\{1,2\}$,
\begin{equation}
1\geq(\delta-\beta_{i})^{2}\ \Leftrightarrow\ 1-\delta^{2}-\beta
_{i}^{2}\geq-2\delta\beta_{i}\ \Leftrightarrow\ \delta^{2}%
\leq\delta\beta_{i}, \label{eq:aux}%
\end{equation}
where the last inequality follows from the second one by noting that $\beta_i^2 = \delta^2 +1$. 

Suppose that $\lambda=1$ and let $\delta >0$. From \eqref{eq:aux},
$|1|\geq|\delta-\beta_{2}|,$ implying by Theorem \ref{lcs2x2} that $\beta_{2}$ is not an interior
L-eigenvalue of $H_{\delta}$. On the other hand, $H_{\delta}$ has no
boundary L-eigenvalues. Hence, the only L-eigenvalue of $H_{\delta}$ is
$\beta_{1}$ whose distance from  $1$ is at least $2$, regardless of the value of $\delta
>0$.

With a similar argument, we can see that, for $\delta <0$,  the only L-eigenvalue of
$-H_{\delta}$ is $\beta_{2}$ whose distance from  $-1$ is at least $2$, regardless of
the value of $\delta<0.$

Thus, for each $\lambda\in\{1,-1\}$, there is some $\delta\in\mathbb{R}$
such that one of the matrices $H_{\delta}$ or $-H_{\delta}$ has no
L-eigenvalues arbitrarily close to $\lambda.$
\end{proof}

\begin{lemma}
\label{E12+E21} Suppose that $\phi:W_{2}\rightarrow W_{2}$ is a linear map
that preserves the L-spectrum. Then
$$\phi(E_{12}+E_{21}) = E_{12}+E_{21} \quad \text{or}\quad \phi(E_{12}+E_{21})= -(E_{12}+E_{21}).$$
\end{lemma}

\begin{proof}
Let $H:=E_{12}+E_{21}$. 
By Corollary \ref{spect-cases} and Corollary \ref{rr1}, we have $\sigma_{\mathcal{K}}(H)= \sigma_{\mathcal{K}}(-H)=\{-1,1\}.$

We start by proving that $1$ and $-1$ are not interior L-eigenvalues of
$\phi(H).$ To show this fact, suppose first that $\lambda\in\{-1,1\}$ is an
interior L-eigenvalue of $\phi(H).$ Then, since by Corollary \ref{trace},
$\operatorname*{tr}(\phi(H))=\operatorname*{tr}(H)=0,$ and interior
L-eigenvalues are standard eigenvalues, $\phi(H)$ has distinct standard
eigenvalues $1$ and $-1.$

We first show that the entry in position (1,1) of $\phi(H)$ is different from
$\lambda.$ This is clear by Theorem \ref{lcs2x2},  if the entry in position $(1,2)$ of $\phi(H)$ is nonzero.  If the entry in position $(1,2)$ of $\phi(H)$ is zero, then
$\phi(H)$ is a lower triangular matrix with main diagonal entries $1$
and $-1$, and the $(2,1)$ entry of $\phi(H)$ has modulus $2$ (since by
Corollary \ref{trace}, the modulus of the anti-trace is preserved). Then, the
entry in position $(1,1)$ of $\phi(H)$ is different from $\lambda,$ as
otherwise, by Theorem \ref{lcs2x2}, $\lambda$ would not be an interior
L-eigenvalue of $\phi(H)$. 

By Lemma \ref{lintlim}, any matrix $B$ in a
sufficiently small neighborhood of $\phi(H)$ has an interior L-eigenvalue
arbitrarily close to $\lambda.$ By the continuity of $\phi^{-1},$ and since
$\phi^{-1}$ preserves the L-spectrum, any matrix  in a sufficiently small
neighborhood of $H$ has an L-eigenvalue arbitrarily close to $\lambda$, which is
impossible by Lemma \ref{E1221pert}.

Thus, $1$ and $-1$ are not interior L-eigenvalues of $\phi(H).$ By Corollary
\ref{l1}, neither $1$ nor $-1$ is an interior L-eigenvalue of $-\phi(H).$ Since
$\sigma_{\mathcal{K}}(H)=\sigma_{\mathcal{K}}(-H)=\{1,-1\}$, we conclude that $1$ and $-1$ are boundary
L-eigenvalues of both $\phi(H)$ and $-\phi(H).$ By Corollary \ref{2bd}, there
are $x,y\in\mathbb{R}$ such that
\[
\text{1) }\phi(H)=\left[
\begin{array}
[c]{cc}%
x & y\\
2-y & -x
\end{array}
\right]  \quad\text{ or}\quad\text{ 2) }\phi(H)=\left[
\begin{array}
[c]{cc}%
x & y\\
-2-y & -x
\end{array}
\right]  .
\]

Suppose that Case 1 holds. Then, by Condition 3 of Theorem \ref{lcs2x2},
applied to both $\phi(H)$ and $-\phi(H)$, we have
\begin{align*}
x+y  &  =-x+2-y \;\text{ and}\\
x-y  &  =-x-(2-y),
\end{align*}
implying that
\[
x=0\text{ and }y=1.
\]
A similar argument applied to Case 2 yields $x=0$ and $y=-1$. Thus, the claim follows.
\end{proof}

\section{Proof of the main results}

\label{proofmain}

\ \\
\ \\
\indent{\bf Theorem \ref{tmain1}}

\begin{proof}
Let $\phi:W_{2}\rightarrow W_{2}$ be a linear map that preserves the
L-spectrum. By Corollary \ref{trace}, either $A$ and $\phi(A)$ have the same
anti-trace for all $A\in W_{2}$, or $A$ and $\phi(A)$ have opposite
anti-traces for all $A\in W_{2}.$ When proving Theorem \ref{tmain1}, we only
consider the case in which $\phi$ preserves the anti-trace. The case when the
anti-trace of $A$ and $\phi(A)$ are opposite for all $A\in W_{2}$ can be
obtained by considering the orthogonal similarity via the matrix $T=\left[
-1\right]  \oplus\left[  1\right]  .$ More precisely, assume that $A$ and
$\phi(A)$ have opposite anti-traces. Then, $\pi(A)=T\phi(A)T,$ for $A\in
W_{2},$ is a linear map that preserves the anti-trace and symmetry, and, taking into
account Corollary \ref{rr1}, $\pi$ preserves the L-spectrum if and only if
$\phi$ does. Hence, by the result that we next show, $\pi$ preserves the
L-spectrum if and only if there is some $P\in M_{2},$ as in (\ref{PQ}), such
that $\pi(A)=PAP^{-1}$ for any $A\in W_{2}$, that is, $\phi(A)=(TP)A(TP)^{-1}$
for any $A\in W_{2}.$ Thus, the claim follows with $Q=TP.$

Necessity: Suppose that $\phi$ preserves the anti-trace. For $u,v\in
\mathbb{R},$ let%
\[
P(u,v):=\left[
\begin{array}
[c]{cc}%
u & v\\
v & u
\end{array}
\right].%
\]

Case 1: Assume that $W_{2}=S_{2}.$ By Lemmas \ref{E22-E11} and \ref{E12+E21}, we have
$\phi(E_{11})=E_{11},$ $\phi(E_{22})=E_{22}$, and $\phi(E_{12}+E_{21}%
)=E_{12}+E_{21}$. Thus, $\phi(A)=PAP^{-1}$ for all $A\in S_{2}$, where
$P=P(1,0)=I$.

Case 2: Assume now that $W_{2}=M_{2}$. By Lemma \ref{E22-E11}, for some $a\leq 0$, we have
\begin{align*}
\phi(E_{11})  &  =%
\begin{bmatrix}
1-a & \mp\sqrt{a^{2}-a}\\
\pm\sqrt{a^{2}-a} & a
\end{bmatrix}
=:\left[
\begin{array}
[c]{cc}%
\alpha^{2} & -\alpha\beta\\
\alpha\beta & -\beta^{2}%
\end{array}
\right] \\
&  =P(\alpha,\beta)E_{11}P^{-1}(\alpha,\beta).
\end{align*}
%with $\alpha^{2}=1-a$ and $\beta^{2}=-a$. 
Without loss of generality, we  assume $\alpha \geq 0$, implying $\alpha \geq 1$ since $\alpha^2 = 1-a$ and $a\leq 0$.

By Lemma \ref{lE21} and taking into account that $\phi$ preserves the anti-trace, for some $b\geq 0$, we have
\begin{align*}
\phi(E_{21})  &  =\left[
\begin{array}
[c]{cc}%
\pm\sqrt{b^{2}+b} & -b\\
b+1 & \mp\sqrt{b^{2}+b}%
\end{array}
\right]  =: \allowbreak\left[
\begin{array}
[c]{cc}%
\gamma\delta & -\delta^{2}\\
\gamma^{2} & -\gamma\delta
\end{array}
\right] \\
&  =P(\gamma,\delta)E_{21}P^{-1}(\gamma,\delta).
\end{align*}
As above, we assume $\gamma \geq 0$, implying $\gamma \geq 1$.  

Then
\begin{align*}
\phi(E_{11}+E_{21})  &  =\left[
\begin{array}
[c]{cc}%
\alpha^{2} & -\alpha\beta\\
\alpha\beta & -\beta^{2}%
\end{array}
\right]  +\left[
\begin{array}
[c]{cc}%
\gamma\delta & -\delta^{2}\\
\gamma^{2} & -\gamma\delta
\end{array}
\right] \\
&  =\allowbreak\left[
\begin{array}
[c]{cc}%
\alpha^{2}+\gamma\delta & -\alpha\beta-\delta^{2}\\
\alpha\beta+\gamma^{2} & -\beta^{2}-\gamma\delta
\end{array}
\right]  .
\end{align*}
Since, by Lemma \ref{singular}, $\phi(E_{11}+E_{21})$ is singular, we have
\[
\det(\phi(E_{11}+E_{21}))=\allowbreak\left(  \alpha\gamma-\beta\delta\right)
\left(  \beta\gamma-\alpha\delta\right)  =0.
\]
Note that $\alpha\gamma-\beta\gamma\neq0$, as otherwise $(\alpha\gamma
)^{2}=(\beta\delta)^{2}$, or equivalently, $a=1+b$, a contradiction since
$a\leq0$ and $1+b>0.$ Thus, 
\begin{equation}\label{aldel}
\beta\gamma = \alpha\delta,
\end{equation}
 implying%
\[
0=(\alpha\delta)^{2}-(\beta\gamma)^{2}=(1-a)b+a(1+b)=a+b.
\]
Hence, $a=-b$ which yields $\alpha=\gamma$. Since $\alpha$ and $\gamma$ are nonzero,  from \eqref{aldel} we get $\beta = \delta .$ 
Now let
$P:=P(\alpha,\beta)$. Then,
\[
\phi(E_{11})=PE_{11}P^{-1}\quad\text{and}\quad\phi(E_{21})=PE_{21}P^{-1},
\]
implying
\begin{align*}
\phi(E_{22})  &  =I-\phi(E_{11})=I-PE_{11}P^{-1}\\
&  =P(I-E_{11})P^{-1}=PE_{22}P^{-1}.
\end{align*}
Moreover, taking into account Lemma \ref{E12+E21} and the fact that $\phi$ preserves the anti-trace, we have
\[
\phi(E_{12}+E_{21})=E_{12}+E_{21}=P(E_{12}+E_{21})P^{-1}.
\]
Thus, since $\phi(A)=PAP^{-1}$ for all the matrices $A$ in a basis for $M_{2}%
$, we have $\phi(A)=PAP^{-1}$ for all $A\in M_{2}.$

Sufficiency: Let $A\in W_{2}$ and let $P$ be as in (\ref{PQ}) with $\alpha
^{2}-\beta^{2}=1$. We assume that $\alpha>0$ as, otherwise, since
$PAP^{-1}=(-P)A(-P)^{-1},$ we may consider $-P$ instead of $P$. It is enough to
prove  $\sigma_{\mathcal{K}}(A)\subseteq\sigma_{\mathcal{K}}(\phi(A))$,
since by applying this result to $\phi^{-1}$, we get $\sigma_{\mathcal{K}}%
(\phi(A))\subseteq\sigma_{\mathcal{K}}(A)$. (Note that $\phi^{-1}
(A)=P^{-1}AP,$ where $P^{-1}$ still has the form of $P$ in (\ref{PQ}), with
$\beta$ replaced by $-\beta.$)

We show that if $(\lambda,x)$ is an L-eigenpair of $A$, then $(\lambda,Px)$ is
an L-eigenpair of $\phi(A)=PAP^{-1}$. For this purpose, we start by proving
two facts.  First, $P$ preserves the
Lorentz cone, that is, if $x\in\mathcal{K}$, then $Px\in\mathcal{K}$. Second,
$P$ preserves orthogonality, that is, if $x^{T}y=0$, then $(Px)^{T}(Py)=0,$
for $x,y\in\mathcal{K}$.

Let $x=[x_{1}\;x_{2}]^{T}\in\mathcal{K}$ and
\begin{equation}
\lbrack z_{1}\;z_{2}]^{T}:=Px=[x_{1}\alpha+x_{2}\beta,x_{1}\beta+x_{2}%
\alpha]^{T}.\label{z}%
\end{equation}
Then, $Px\in\mathcal{K}$ if and only if
\[
|z_{1}|=|x_{1}\alpha+x_{2}\beta|\leq x_{1}\beta+x_{2}\alpha=z_{2}.
\]
Since $|\beta|<\alpha$ and $|x_{1}|\leq x_{2},$ it follows that $z_2=x_{1}%
\beta+x_{2}\alpha\geq0.$ Also, because of 
\begin{equation}
z_{1}^{2}-z_{2}^{2}=x_{1}^{2}-x_{2}^{2}\leq0,\label{t1}%
\end{equation}
 we get that $Px\in\mathcal{K}$.

Now note that, if $x$ and $y$ are nonzero orthogonal vectors in $\mathcal{K}$,
then they lie on the boundary of $\mathcal{K}$. More specifically, one is a
positive multiple of $[1$ $1]^{T}$ and the other one is a positive multiple of
$[-1$ $1]^{T}.$ Since
\[
P[1 \ 1]^{T}=[\alpha+\beta,\alpha+\beta]^{T}\quad\text{and}\quad P[-1 \ 1]^{T}
=[-\alpha+\beta,\alpha-\beta]^{T}%
\]
are orthogonal, it follows that $P$ also preserves orthogonality.

Suppose that $(\lambda,x)$ is an L-eigenpair of $A$, that is, 
\[
x\neq 0, \quad x\in \mathcal{K},\quad(A-\lambda I)x\in\mathcal{K},\quad\text{ and }\quad x^{T}(A-\lambda
I)x=0.
\]
Since $P$ is invertible, we have $Px\neq 0$. Moreover, as  $P$ preserves the Lorentz cone, we have $y:=Px\in\mathcal{K}$ and
\[
(\phi(A)-\lambda I)y=P(A-\lambda I)P^{-1}Px=P[(A-\lambda I)x]\in\mathcal{K}.
\]
From the orthogonality of $x$ and $(A-\lambda I)x$ and the fact that $P$
preserves orthogonality, it follows that $y^{T}(\phi(A)-\lambda I)y=0.$ Thus,
$(\lambda,Px)$ is an L-eigenpair of $\phi(A).$
\end{proof}

\bigskip
{\bf Corollary \ref{cornature}}

\begin{proof}%(Proof of Corollary \ref{cornature})
By Theorem \ref{tmain1}, and arguing as in its proof, we may assume that
$\phi$ preserves the anti-trace, that is, $\phi(A)=PAP^{-1}$ for $P$ as in
(\ref{PQ}) with $\alpha^{2}-\beta^{2}=1.$ Moreover, we may assume
that $\alpha>0,$ as otherwise we consider $-P$ instead of $P$.

Assume that $(\lambda,x)$ is an L-eigenpair of $A,$ with $x=[x_{1}\;x_{2}]^{T}%
$. Let $z=[z_{1}\;z_{2}]^{T}$ be as in (\ref{z}). It was shown in the
sufficiency part of the proof of  Theorem \ref{tmain1} that $(\lambda,z)$ is an
L-eigenpair of $\phi(A).$ Since, by \eqref{t1}, $|x_{1}|<x_{2}$ if and only if
$|z_{1}|<z_{2}$, it follows that $z$ is an  L-eigenvector of $\phi(A)$ in the interior of $\mathcal{K}$
if and only if $x$ is an  L-eigenvector  of $A$ in the interior of $\mathcal{K}$. Since $A$ and $\phi(A)$
have the same L-spectrum, the claim follows. 
\end{proof}

\section{Conclusions}

\label{conclusions} Let $M_{n}$ denote the space of $n\times n$ real matrices
and $S_{n}$ denote the subspace of $M_{n}$ formed by the symmetric matrices. In
this paper, for $W_{2}\in\{M_{2},S_{2}\},$ we described the linear maps
$\phi:W_{2}\rightarrow W_{2}$ that preserve the Lorentz spectrum (L-spectrum
for short), that is, those maps $\phi$ for which  $A$ and $\phi(A)$ have the same
L-spectrum for all $A\in W_{2}.$ We have shown that
$\phi(A)=PAP^{-1}$, where $P$ is a matrix with a certain 
structure. In the case $W_{2}=S_{2}$, $P$ is a diagonal orthogonal matrix. In addition, we proved that such
preservers on $W_{2}$ do not change the nature (interior or boundary) of the L-eigenvalues.

In the case $n\geq3$, the characterization of the linear maps
$\phi:W_{n}\rightarrow W_{n}$ that preserve the L-spectrum and are standard
was given in \cite{Bueno2021}. (See \cite{Seeger3} in which the case
$W_{n}=M_{n}$ was also studied.) Recall that a linear map $\phi:W_{n}
\rightarrow W_{n}$ is said to be standard if there exist matrices $P,Q\in
M_{n}$ such that $\phi(A)=PAQ$ for all $A\in W_{n}$ or $\phi(A)=PA^{T}Q$ for
all $A\in W_{n}.$ In \cite{Bueno2021}, a conjecture was made that all maps
$\phi:W_{n}\rightarrow W_{n}$ that preserve the L-spectrum are, in fact,
standard, as has been shown here to happen for $n=2$. We also have seen here
that these preservers on $W_{2}=S_{2}$ have the same form as the standard ones
on $S_{n}$, for $n\geq3$. However, if $W_{2}=M_{2}$, they have a more general
form than those on $M_{n}$ for $n\geq3$.

Contrary to what happens when $n\geq3$, the Lorentz cone in
$\mathbb{R}^{n}$ with $n=2$ is a polyhedral cone which is a rotation of the
Pareto cone. Thus, our characterization of the linear maps $\phi
:M_{2}\rightarrow M_{2}$ also follows from the characterization of the linear
maps that preserve the Pareto spectrum \cite{zadsha}. However, we gave here an
independent proof hoping that it gives tools that may be helpful in proving
the still open conjecture stated in \cite{Bueno2021} that any linear
preservers of the L-spectrum on $S_{n}$ or $M_{n}$, for $n\geq3$, are standard
maps, which, together with the results in that reference, would complete the
description of such linear preservers.

\bigskip
{\bf Acknowledgments} We would like to thank Professor K. C. Sivakumar for proposing us the study of linear preservers of the Lorentz spectrum. This work is a continuation of the collaboration with him that resulted  in  publication \cite{Bueno2021}. We would also like to thank the anonymous referees of this paper for their comments and suggestions.

\end{document}